\documentclass[12pt]{amsart}

\usepackage{amssymb, amsmath, amsthm} 

\usepackage[pdftex, hyperindex=true]{hyperref}

\hypersetup{
 colorlinks = true,
 linkcolor = black,
 urlcolor = black,
 citecolor = black
}

\newcommand{\bbZ}{\mathbb{Z}}

\newcommand{\rmA}{\mathrm{A}}
\newcommand{\rmB}{\mathrm{B}}

\newcommand{\rmE}{\mathrm{E}}

\newcommand{\rmH}{\mathrm{H}}
\newcommand{\rmI}{\mathrm{I}}

\newcommand{\rmU}{\mathrm{U}}
\newcommand{\rmZ}{\mathrm{Z}}

\DeclareMathOperator{\id}{id}

\DeclareMathOperator{\Ab}{Ab}
\DeclareMathOperator{\Cl}{Cl}
\DeclareMathOperator{\Gr}{Gr}
\DeclareMathOperator{\Or}{Or}
\DeclareMathOperator{\Pq}{Pq}

\DeclareMathOperator{\Aut}{Aut}
\DeclareMathOperator{\Hom}{Hom}
\DeclareMathOperator{\Ker}{Ker}

\newtheorem{theorem}{Theorem}
\newtheorem{proposition}[theorem]{Proposition}
\newtheorem{corollary}[theorem]{Corollary}
\newtheorem{lemma}[theorem]{Lemma}

\theoremstyle{definition}
\newtheorem{definition}[theorem]{Definition}
\newtheorem{example}[theorem]{Example}

\newtheorem{remark}[theorem]{Remark}

\newtheorem{question}[theorem]{Question}

\numberwithin{theorem}{section}
\numberwithin{equation}{section}
\numberwithin{figure}{section}

\title{Groups, conjugation and powers}

\author{Markus Szymik}

\address{School of Mathematics and Statistics, 
The University of Sheffield,
Sheffield S3 7RH,
United Kingdom, and\newline
\indent Department of Mathematical Sciences,
NTNU Norwegian University of Science and Technology,
7491 Trondheim,
Norway}

\email{m.szymik@sheffield.ac.uk, markus.szymik@ntnu.no}

\author{Torstein Vik}

\address{Department of Mathematical Sciences,
NTNU Norwegian University of Science and Technology,
7491 Trondheim,
Norway}

\email{torstvik@stud.ntnu.no}

\begin{document}


\begin{abstract}
We introduce the notion of the power quandle of a group, an algebraic structure that forgets the multiplication but keeps the conjugation and the power maps. Compared with plain quandles, power quandles are much better invariants of groups. We show that they determine the central quotient of any group and the center of any finite group. Any group can be canonically approximated by the associated group of its power quandle, which we show to be a central extension, with a universal property, and a computable kernel. This allows us to present any group as a quotient of a group with a power-conjugation presentation by an abelian subgroup that is determined by the power quandle and low-dimensional homological invariants.
\end{abstract}

 

\maketitle

\thispagestyle{empty}


\section*{Introduction}

Groups are one of the most important algebraic structures known today. While fundamental for the mathematical description of symmetry, the group axioms themselves reflect this purpose only poorly. The usual presentation of the theory of groups in terms of a multiplication, inverses, and a unit is actually just one of many. After all, the~$n$--ary operations on groups are given by the elements of the free group on~$n$ generators, which have been studied for more than a century.

The attempt to found a theory of symmetry on conjugation instead of multiplication has led to the algebraic theory of quandles, which has many applications, most prominently in the theory of knots and braids~\cite{Joyce, Matveev, Brieskorn, Kauffman, Fenn--Rourke, Elhamdadi--Nelson}. Still, the relationship between groups and quandles is comparatively loose. Every group has an underlying quandle, but this invariant is useless in distinguishing abelian group, for instance, where conjugation is trivial. Conversely, every quandle has an associated group, but this group will only be finite in trivial cases, destroying all hopes of reproducing groups from quandles. 

In this paper, we introduce a new algebraic structure, namely{\it~power quandles}, that is made to remedy these defects of plain quandles by taking into account not only the conjugation in groups, but also all~$0$--ary and~$1$--ary operations that we have: the unit~$e$ and the power maps~$g\mapsto g^n$ for any integer~$n$. The precise statement of the axioms is Definition~\ref{def:power_quandle} in Section~\ref{sec:power_quandles}, which also contains some elementary examples. The short Section~\ref{sec:abelian} discusses{\it~abelian} power quandles, opening the gate for abelianization and homology.

Section~\ref{sec:forgetful} contains our results on the forgetful functor that sends each group~$G$ to its underlying power quandle~$\Pq(G)$. We show in Theorem~\ref{thm:central_quotient} that an isomorphism~$\Pq(G)\cong\Pq(H)$ of power quandles implies that the center quotients are isomorphic:~$G/\rmZ(G)\cong H/\rmZ(H)$. The centers themselves are also isomorphic, at least when the groups are finite; in general, we can still deduce that the centers have the same cardinality~(see Theorem~\ref{thm:center}). Both results together place the question, whether finite groups are determined by their underlying power quandles, in the context of Baer's problem~\cite{Baer} of describing all groups with preassigned central and central quotient group~(see Question~\ref{q:forgetful}).

In the Section~\ref{sec:adjoint}, we study the left-adjoint~$P\mapsto\Gr(P)$ of the forgetful functor~$\Pq$. When trying to reconstruct a group~$G$ from its power quandle~$\Pq(G)$, the canonical candidate is~$\Gr\Pq(G)$, which comes with a surjection~$\Gr\Pq(G)\to G$, the co-unit of the adjunction. We show that it displays~$\Gr\Pq(G)$ as a{\it~central} extensions of~$G$ that we characterize by a universal property~(see Theorems~\ref{thm:central} and~\ref{thm:universal}), and we compute the abelian kernel, which measures the failure of~$\Gr\Pq(G)\to G$ to be an isomorphism, in terms of the power quandle and{\it~homological} information about the group~$G$, namely the abelianization~$\rmH_1(G;\bbZ)$ and the Schur multiplier~$\rmH_2(G;\bbZ)$~(see Theorem~\ref{thm:kernel}).

In the final Section~\ref{sec:image}, we show that finite groups that are of the form~\hbox{$G\cong\Gr(P)$} are determined by their underlying power quandles~(see Theorem~\ref{thm:pq-gen}). We call such groups{\it~pq-generated}. Examples are not difficult to come by. In fact, in Question~\ref{q:pq-gen} we ask whether all finite groups are pq-generated, and Theorem~\ref{thm:pq-gen} relates this directly to Question~\ref{q:forgetful}.

Some parts of this work have been formalized in the Lean proof assistant~\cite{Lean,Code}.


\section{Power quandles}\label{sec:power_quandles}

In this section, we define power quandles as algebraic structures that have operations with properties familiar from the conjugation, the unit, and the power maps in groups~(see Definition~\ref{def:power_quandle}). As a preparation, we recall the two more basic structures that abstract these ideas individually, before we merge them into one.

A{\it~rack}~$(R,\rhd)$ is a pair consisting of a set~$R$ together with a binary operation~\hbox{$\rhd\colon R\times R\to R$} such that, for all elements~$a$ in~$R$, the left-multiplication~$\lambda_a\colon R\to R,\,b\mapsto a\rhd b$ is an automorphism of~$(R,\rhd)$. The fact that~$\lambda_a$ is a morphism of racks can be rephrased, as a formula, as
\begin{equation}\label{eq:rack}
a\rhd(b\rhd c)=(a\rhd b)\rhd(a\rhd c)
\end{equation}
for all elements~$b$ and~$c$ in~$R$. 
By definition, a rack brings its own symmetries, the left-multiplications. However, all ~{\it natural} automorphisms are generated by the canonical automorphism~\hbox{$a\mapsto a\rhd a$}~(see~\cite[Thm.~5.4]{Szymik:1}). A{\it~quandle} is a rack such that the canonical automorphism is trivial:
\begin{equation}
a\rhd a = a
\end{equation}
for all~$a$. 
A{\it~unit} for a rack or quandle is an element~$e$ that is{\it~fixing},
\begin{equation}\label{eq:fixing}
e\rhd b = b,
\end{equation}
and{\it~fixed},
\begin{equation}\label{eq:fixed}
a\rhd e = e,
\end{equation}
for all elements~$a$ and~$b$~(see~\cite[Sec.~2.2]{Lawson--Szymik}). This defines~{\it unital racks} and~{\it unital quandles}. By~\eqref{eq:fixed}, the left-multiplications are morphisms of unital racks or quandles. Every group has an underlying unital quandle, with~$a\rhd b=aba^{-1}$ and the unit~$e$.


An{\it~action}~$\pi$ of the multiplicative abelian monoid~$(\mathbb{Z},\times)$ of the integers on a set~$X$ is given by a family~$(\,\pi^n\colon X\to X\,|\,n\in\bbZ\,)$ of unary operations on~$X$ such that
\begin{equation}
\pi^1(x)=x
\end{equation}
and
\begin{equation}
\pi^m(\pi^n(x))=\pi^{mn}(x)
\end{equation}
for all elements~$x$ of~$X$. Every group has an underlying~$(\mathbb{Z},\times)$--action, defined by~\hbox{$\pi^n(a)=a^n$}.


We are now ready for the main definition of this paper.

\begin{definition}\label{def:power_quandle}
A{\it~power quandle}~$(P,\rhd,\pi,e)$ is given by a unital quandle~$(P,\rhd,e)$ together with an action of the multiplicative abelian monoid~$(\mathbb{Z},\times)$ of the integers such that these structures are compatible in the following sense:\\
First, we have
\begin{equation}
\pi^0(a)=e
\end{equation}
for all~$a$ in~$P$.\\
Second, the left-multiplications are compatible with the~$\pi$--action:
\begin{equation}\label{eq:compatible_with_action}
a\rhd\pi^n(b)=\pi^n(a\rhd b)
\end{equation}
for all~$a$ and~$b$ in~$P$ and integers~$n$.
This second condition implies that we get a map
\[
\lambda\colon P\longrightarrow\Aut(P,\rhd,\pi,e),\,a\longmapsto \lambda_a
\]
into the automorphism group of the structure.\\ 
Third and last, we require that this map~$\lambda$ is compatible with the canonical operations on the right hand side, given by~$f\rhd g=fgf^{-1}$ and~$\pi^n(f)=f^n$ and the identity~$\id_P$. In formulas, this is equivalent to~\eqref{eq:rack},~\eqref{eq:fixing}, and the new axiom
\begin{equation}
\pi^n(a)\rhd b = a\rhd^n b,
\end{equation}
for all elements~$a$ and~$b$ in~$P$ and all integers~$n$, where~$a\rhd^n b=\lambda^n_a(b)$ is the result of applying the left-multiplication~$\lambda_a(-)=a\rhd(-)$ operator~$n$ times to~$b$.
\end{definition}

\begin{example}\label{ex:underlying_group}
If~$G$ is a group, a power quandle structure on the same underlying set can be defined by the equation~$a\rhd b=aba^{-1}$ and~\hbox{$\pi^n(a)=a^n$} and the unit~$e$. This structure has already implicitly been used in Definition~\ref{def:power_quandle}, where the automorphism group of a power quandle was equipped with a power quandle structure in this way.
\end{example}

Power quandles form a category with the obvious morphisms, and the preceding example can be rephrased to say that there is a `forgetful' functor~$\Pq\colon G\mapsto\Pq(G)$ from the category of groups to the category of power quandles. We will say more about this functor and its adjoint in the following sections. For now, we restrict ourselves to give some examples that show, among other things, that the category of power quandles is much larger than the category of groups.

\begin{example}\label{ex:underlying_abelian}
A power quandle with~$a\rhd b =b$ for all~$a$ and~$b$ is essentially the same as a set~$P$ together with a~$(\mathbb{Z},\times)$--action such that~$\pi^0$ is constant. Such an action is given by the constant~$e$, an involution~$\pi^{-1}$, and a family of self-maps~$\pi^p$, one for each prime number~$p$, such that all of these commute with each other.
\end{example}

\begin{example}
There are more power quandles than groups. More precisely, the forgetful functor is not essentially surjective: there are power quandles that are not isomorphic to the power quandle underlying a group. For instance, there is only one group of order~$3$, up to isomorphism, and it is abelian, so the underlying power quandle is given as in the previous Example~\ref{ex:underlying_abelian}. On the other hand, there are clearly more such power quandle structures on any set with three elements. Specifically, the one with~$\pi^n=\id$ for~$n\not=0$ is different than the one coming from the group.
\end{example}

\begin{remark}
Here is another aspect in which quandles that underlie groups are special. Let us write~$\kappa_a(b)=(ab)^{-1}=b^{-1}a^{-1}$. It is straightforward to verify that we have~$\kappa_a\kappa_a(b)=aba^{-1}=\lambda_a(b)$. This shows that all left-multiplications are {\it even} permutations in quandles that underlie groups.
\end{remark}

\begin{example}
There are more power quandle morphisms than morphisms of groups. More precisely, the forgetful functor is not full: there are groups~$G$ and~$H$ and maps~$G\to H$ that are not morphisms of groups but compatible with the underlying power structures. Specifically, let~\hbox{$V\cong\bbZ/2\times\bbZ/2$} be the Klein group. There are~$16$ group morphisms~$V\to V$, but every map~$V\to V$ that sends the unit~$e$ to itself is compatible with the power structure, and there are~$64$ of those.
\end{example}


\begin{example}\label{ex:orbits}
Recall that, given any quandle~$Q$, the set~$\Or(Q)$ of orbits is defined as the set of equivalence classes~$[a]$ of elements~$a$ of~$Q$ for the equivalence relation generated by~$a\rhd b\sim b$ for all~$a$ and~$b$. If~$P$ is a power quandle, then~$\Or(P)$ inherits a~$(\mathbb{Z},\times)$--action via~$\pi^n([a])=[\pi^n(a)]$. This is well-defined by~\eqref{eq:compatible_with_action}, and it equips~$\Or(Q)$ with the structure of a power quandle for which the conjugation~$\rhd$ is trivial, as in Example~\ref{ex:underlying_abelian}. Of particular interest is the case of the power quandle~$P=\Pq(G)$ of a group~$G$: then~$\Or\Pq(G)$ is the set~$\Cl(G)$ of conjugacy classes of elements, with its inherited power structure. This additional information is very relevant for representation theory~(see~\cite{Meir--Szymik}, for instance). The power maps are often presented alongside the character table, as in GAP~\cite{GAP}. 
\end{example}


\section{Abelian power quandles}\label{sec:abelian}

For every algebraic theory one can ask for an additional abelian group structure on the models such that all operations are homomorphisms~\cite{Szymik:2}. For groups, these are just the abelian groups. For plain or unital quandles, such a structure is equivalent to a module over the ring~$\bbZ[q^{\pm1}]$ of integral Laurent polynomials, the quandle structure being given by~\hbox{$a\rhd b=(1-q)a+qb$}. For power quandles, we have the following result.

\begin{lemma}
Let~$P$ be a power quandle and an abelian group such that the operations~$\rhd$ and~$\pi^n$ are homomorphisms of groups and~$e=0$ is zero. Then the equation~$a\rhd b=b$ holds for all~$a$ and~$b$ in~$P$.
\end{lemma}

\begin{proof}
We start by noting~$\pi^0(a)=e=0$ for all elements~$a$ in the power quandle~$P$. Equation~\eqref{eq:compatible_with_action} then gives
\[
a\rhd 0=a\rhd\pi^0(0)=\pi^0(a\rhd 0)=\pi^0(0)=0.
\] 
It follows that we have
\[
a\rhd b=(a\rhd 0)+(0\rhd b)=0\rhd b
\]
for all~$a$ and~$b$. Since the right hand side is visibly independent of~$a$, it must be equal to~$b\rhd b=b$.
\end{proof}

It follows that the quandle structure of abelian power quandles is automatically trivial and only the power operations matter. The situation can be described in more elaborate terms as follows.

\begin{proposition}
The category of abelian group objects in power quandles is equivalent to the category of abelian groups with linear actions of the monoid~$(\mathbb{Z},\times)$, and hence to the category of modules over the corresponding monoid ring. 
\end{proposition}

\begin{proposition}
The abelianization functor,~i.e., the left-adjoint to the forgetful functor from the category of abelian power quandles to the category of all power quandles, sends a power quandle~$P$ to~$\bbZ\Or(P)$, the free abelian group on the set~$\Or(P)$ of orbits of~$P$, with trivial conjugation and the induced power structure as in Example~\ref{ex:orbits}.
\end{proposition}

This basic understanding of abelian power quandles is enough to develop the Quillen homology theory for power quandles by deriving the abelianization functor. We will not do this here and refer to~\cite{Szymik:3}, where one of us treated the case of plain quandles, without power structure.


\section{The forgetful functor}\label{sec:forgetful}

This section contains our results on the forgetful functor that assigns to each group its underlying power quandle as in Example~\ref{ex:underlying_group}.

\begin{theorem}\label{thm:central_quotient}
Groups~$G$ with isomorphic underlying power quandles have isomorphic central quotients~$G/\rmZ(G)$.
\end{theorem}

\begin{proof}
Let~$f\colon G\to H$ be a bijection between groups that is compatible with the power quandle structures on both sides. We would like to produce an bijection~\hbox{$G/\rmZ(G)\to H/\rmZ(H)$} that is compatible with the group structures on both sides. We first claim that the composition~\hbox{$G\to H\to H/\rmZ(H)$} of~$f$ with the canonical projection is a morphism of groups: Let~$a$ and~$b$ be any elements of the group~$G$. Using the relation~\hbox{$(xy)\rhd z=x\rhd(y\rhd z)$}, which holds in the underlying quandle of any group, and the fact that~$f$ is compatible with the power quandle structures, we compute
\begin{align*}
	f(ab)\rhd f(g)
	&=f((ab)\rhd g)\\
	&=f(a\rhd(b\rhd g))\\
	&=f(a)\rhd(f(b)\rhd f(g)))\\
	&=(f(a)f(b))\rhd f(g).
\end{align*}
Since~$f$ is surjective, this means that~$f(ab)\rhd h=(f(a)f(b))\rhd h$ for all~$h$ in~$H$. Then the conjugations by~$f(ab)$ and by~$f(a)f(b)$ are equal on~$H$, so that these elements differ by an element in the center~$\rmZ(H)$, proving the claim.

We now have a morphism~$G\to H/\rmZ(H)$ of groups that is surjective as a composition of two surjections. It suffices to show that the kernel is~$\rmZ(G)$. An element~$a$ lies in~$\rmZ(G)$ if and only if~$a\rhd g=g~$ for all~$g\in G$. Since~$f$ is an isomorphism of power quandles, this happens if and only if~$f(a)\rhd f(g)=f(a\rhd g)=f(g)$ for all~$g$. As above, since~$f$ is surjective, this is equivalent to~$f(a)\rhd h=h$ for all~$h$ in~$H$, that is~\hbox{$f(a)\in\rmZ(H)$}, and we are done.
\end{proof}

\begin{theorem}\label{thm:center}
Groups with isomorphic underlying power quandles have centers of the same cardinality. If these centers are finite, then they are even isomorphic.
\end{theorem}

\begin{example}
Before we embark on the proof, let us see an example that illustrates the necessity of some hypothesis in the statement. Consider the infinite abelian groups~$\bbZ^r$ of rank~$r$. These are pairwise non-isomorphic as groups. The power quandles can be described as follows. Let~$C_r$ be the set of tuples~$(a_1,\dots,a_r)$ of integers that are co-prime, which means~$\mathrm{gcd}(a_1,\dots,a_r)=(1)$, and satisfy the inequality~\hbox{$(a_1,\dots,a_r)\geqslant(0,\dots,0)$} in the lexicographic order. Then there is an isomorphism~$\bbZ^r\cong\bbZ\,\boxtimes\, C_r$ of power quandles, where for any set~$D$ we set~$\bbZ\,\boxtimes\, D=(\bbZ\times D)/\!\!\sim$, with~\hbox{$(0,x)\sim(0,y)$} the unit for all~$x,y\in D$, trivial conjugation $\rhd$, and~\hbox{$\pi^m[n,x]=[mn,x]$} for all~\hbox{$m,n\in\bbZ$}. The isomorphism sends~$[n,(a_1,\dots,a_r)]$ to~\hbox{$(n\cdot a_1,\dots,n\cdot a_r)$}. Since the isomorphism type of the power quandle~$\bbZ\,\boxtimes\, D$ only depends on the cardinality of~$D$, and~$C_r$ is countably infinite for all~$r\geqslant 2$, we see that the groups~$\bbZ^r$ with~$r\geqslant 2$ all have isomorphic underlying power quandles.
\end{example}

\begin{proof}[Proof of Theorem~\ref{thm:center}]
The center of a group is the set of elements~$a$ such that conjugation by~$a$ is the identity. Similarly, the center of a quandle is the set of elements~$a$ such that left-multiplication with~$a$ is the identity. It is straightforward to check that this subset is a sub-quandle, and it is preserved by the power structure, if present. This means that it does not matter if we compute the center of a group or of its underlying quandle, and groups with isomorphic power quandles have isomorphic centers{\it~as power quandles}. In particular, they have the same cardinality.

It remains to be shown that these centers are isomorphic{\it~as groups} if they are finite. In other words, because centers are abelian, we can assume that the groups in question are abelian, and we are left to show that finite abelian groups are determined by their power structures. 
We indicate an argument for that: Given a finite abelian group~$A$, the power structure determines the cardinalities of the sets
\[
\{\,a\in A\,|\,a^n=e\,\}\cong\Hom(\bbZ/n,A)
\]
for all~$n\geqslant 1$. Now~$A$ itself is a finite product of finite cyclic groups, and the number and structure of the factors can easily be determined from the knowledge of all the cardinalities above~(see~\cite{McHaffey}, for example).
\end{proof}

Theorems~\ref{thm:central_quotient} and~\ref{thm:center} together show that two finite groups with isomorphic underlying power quandles share the same centers and central quotients. The problem which groups share with a given group~$G$ the same center and the same central quotient has already been raised by Baer~\cite{Baer} and is difficult. An obvious upper bound is given by the cohomology~\hbox{$\rmH^2(G/\rmZ(G);\rmZ(G))$} which classifies central extensions of~$G/\rmZ(G)$ with center containing~$\rmZ(G)$, but not necessarily equal to it. Standard examples are the dihedral and quaternion groups of order~$8$. These have non-isomorphism power quandles, though, because the dihedral group has five involutions, whereas the quaternion group has only one. The question whether finite groups are determined by their power quandles remains open. 

\begin{question}\label{q:forgetful}
Are there two non-isomorphic finite groups~$G$ and~$H$ that have isomorphic underlying power quandles?
\end{question}

It is clear from our results that counterexamples need to have a center that is a proper subgroup. A possible approach to answer Question~\ref{q:forgetful} is contained in Section~\ref{sec:image}.

\begin{example}
Recall that two groups are {\it isocategorical} if their complex representation categories are equivalent as monoidal categories. By definition, these groups are difficult to distinguish by means of their representation theory. The smallest examples have order~$64$. From the~267 isomorphism classes of groups of that order, these are the two groups that are named~$\Gamma_{26} a_2$ and~$\Gamma_{26} a_3$ by Hall and Senior~\cite{Hall--Senior}, or{\tt~SmallGroup(64,136)} and{\tt~SmallGroup(64,135)} in~GAP~\cite{GAP}, respectively. Still, these two groups are easily distinguished using their underlying power quandles: the numbers of conjugacy classes of involutions are different~(see also~\cite[Ex.~3.1]{Meir--Szymik}).
\end{example}


\section{The adjoint functor}\label{sec:adjoint}

The forgetful functor~$\Pq$ from the category of groups to the category of power quandles, has a left-adjoint. Namely, we can associate a group~$\Gr(P)$ to every power quandle~$(P,\rhd,\pi,e)$ as follows: The generators of the group~$\Gr(P)$ are elements~$\sigma(a)$, one for each element~$a$ in~$P$, and
\begin{align*}
	\sigma(a\rhd b)&=\sigma(a)\sigma(b)\sigma(a)^{-1},\\
	\sigma(\pi^n(a))&=\sigma(a)^n,\\
	\sigma(e)&=e.
\end{align*}
are the relations, where~$a,b\in P$ and~$n\in\bbZ$. The group~$\Gr(P)$ has the universal property that every morphism from~$P$ into the power quandle~$\Pq(H)$ of a group~$H$ extends uniquely to a morphism of groups from~$\Gr(P)$ to~$H$. In other words, the functor~$\Gr$ is left-adjoint to the forgetful functor~$\Pq$.

The following result will be used in the proof of Theorem~\ref{thm:kernel}.

\begin{lemma}\label{lem:adjoints_commute}
For every power quandle~$P$, the abelianization~$\Ab\Gr(P)$ of the group~$\Gr(P)$ is the quotient of the free abelian group~$\bbZ\Or(P)$ on the set~$\Or(P)$ of orbits of~$P$ by the relations~$n[a]=[\pi^n(a)]$ for all~$a$ in~$P$ and all integers~$n$.
\end{lemma}

\begin{proof}
This follows immediately from the presentation of~$\Gr(P)$ given above. The first relation shows that~$\Ab\Gr(P)$ is a quotient of~$\bbZ\Or(P)$, and the other ones give the indicated relations.
\end{proof}


If~$G$ is any group, we would like to reconstruct it from its power quandle, and the means we have is the canonical surjective morphism~$\epsilon\colon\Gr\Pq(G)\to G$; it sends the generator~$\sigma(a)$ to~$a$. The assignment~\hbox{$\sigma\colon a\mapsto\sigma(a)$} is a set-theoretical section,{\it~normalized} in the sense that~$\sigma(e)$ is the unit, but it is not necessarily a morphism of groups. In the rest of this section, we describe precisely how far off~$G$ the approximation~$\Gr\Pq(G)$ is.

\begin{theorem}\label{thm:central}
For every group~$G$, the kernel of~$\epsilon\colon\Gr\Pq(G)\to G$ lies in the center, so that the group~$\Gr\Pq(G)$ is displayed as a central extension of~$G$ with an abelian kernel.
\end{theorem}

\begin{proof}
As mentioned before, the assignment~$a\mapsto\sigma(a)$ in the inverse direction is not necessarily a morphism of groups. Still, we claim that the map~\hbox{$G\to\Aut(\Gr\Pq(G))$} that sends~$a$ to the conjugation with~$\sigma(a)$ is a morphism of groups. In symbols, given~$a,b$ in~$G$, we need to show that
\begin{equation}\label{eq:tau}
\sigma(ab)\tau\sigma(ab)^{-1}=\sigma(a)\sigma(b)\tau\sigma(b)^{-1}\sigma(a)^{-1}
\end{equation}
for all~$\tau$ in~$\Gr\Pq(G)$. This is trivial if~$\tau$ is the unit. If~$\tau=\sigma(c)$ is a generator, it follows from the compatibility of~$\sigma$ with conjugation, which is the first defining relation in the group~$\Gr\Pq(G)$. Inverting~\eqref{eq:tau} shows that it holds for inverses, and multiplying~\eqref{eq:tau} for two elements~$\tau$ and~$\tau'$ shows that it holds for products, proving the claim by induction.

As a consequence, the morphism~$\Gr\Pq(G)\to\Aut(\Gr\Pq(G))$ that sends~$\sigma(a)$ to the conjugation with~$\sigma(a)$ factors through~$\epsilon$. This means its kernel, which is the center of the group~$\Gr\Pq(G)$, contains the kernel of~$\epsilon$.
\end{proof}


We see that we can present any group~$G$ as a central extension
\[
\rmA(G)\longrightarrow\Gr\Pq(G)\longrightarrow G
\]
with an abelian kernel~$\rmA(G)$. In the following, we will characterize this extension by a universal property, and explain how to compute~$\rmA(G)$ from~$\Pq(G)$ and homological information about the group~$G$. The starting point is again the observation that~$\Gr\Pq(G)\to G$ comes with a canonical normalized set-theoretic section~\hbox{$g\mapsto\sigma(g)$}. This section is not a morphism of groups, in general, but--by construction--it is always a morphism of the underlying power quandles. It turns out that we are facing the universal object with these properties:

\begin{theorem}\label{thm:universal}
Let~$p\colon E\to G$ be a central extension of a group~$G$ together with a normalized section~\hbox{$s\colon G\to E$} that is a morphism of power quandles. There is a unique morphism~\hbox{$\Gr\Pq G\to E$} of~(central) extensions that commutes with the normalized sections.
\end{theorem}

\begin{proof}
Since the group~$\Gr\Pq G$ is generated by the elements~$\sigma(g)$, and the requirements imply~\hbox{$\sigma(g)\mapsto s(g)$}, there is at most one such morphism. Conversely, since the map~\hbox{$g\mapsto s(g)$} is a morphism of power quandles, it extends~(uniquely) to a morphism of groups~$\Gr\Pq G\to E$ with~\hbox{$\sigma(g)\mapsto s(g)$}.
\end{proof}


Our following computation of~$\rmA(G)$ from~$\Pq(G)$ and homological information about the group~$G$ is motivated by the well-known classification of (central) extensions of a given group~$G$ with a given abelian kernel~$A$ in terms of the cohomology of~$G$ with coefficients in~$A$. In fact, the classification usually starts by classifying such extensions{\it~together with a given normalized set-theoretic section} in terms of~$2$--cocycles, which are functions~\hbox{$f\colon G\times G\to A$} satisfying certain axioms. This suggests to look at the{\it~universal} central  extension
\begin{equation}\label{eq:universal}
\rmU(G)\longrightarrow\rmE(G)\longrightarrow G,
\end{equation}
where the abelian group~$\rmU(G)$ is the quotient of the free abelian group on symbols~$u(g,h)$, for~$g,h$ in~$G$, modulo the relations that ensure that~$(g,h)\mapsto u(g,h)$ is a normalized~$2$--cocycle. The abelian group~$\rmU(G)$ represents~$2$--cocycles in the sense that we have natural bijections~$\Hom(\rmU(G),A)\cong\rmZ^2(G;A)$. One can use the bar resolutions to see that there is an exact sequence
\[
0 \longrightarrow \rmH_2(G;\bbZ) \longrightarrow \rmU(G) \longrightarrow \rmI(G) \longrightarrow \rmH_1(G;\bbZ)\longrightarrow 0
\]
of abelian groups, where~$\rmI(G)=\Ker(\bbZ G\to\bbZ)$ is the augmentation ideal of the group ring~$\bbZ G$. The details are not important here; what matters for us is how we can modify such a description to fit when we require the normalized set-theoretic sections to be morphisms of power quandles instead of arbitrary maps.

For instance, if we would ask for the normalized section to be a morphism of groups, then the corresponding kernel would be trivial. The universal such extension is~$\id\colon G\to G$, because any central extension~$E\to G$ with a normalized section that is a morphism of groups has a unique map from~$G$ that makes the obvious diagram commute: it is given by the section.

In our situation~\hbox{$\Gr\Pq G\to G$}, it is unlikely that the group~$\rmU(G)$ is the kernel. Instead, we get a quotient~$\rmA(G)$ of~$\rmU(G)$ by relations that ensure that the normalized section we get for the induced extension is a morphism of power quandles, or at least compatible with conjugation and inversion, such as~\hbox{$u(g,h)+u(gh,g^{-1})=0$}. While this generators-and-relations approach to describing~$\rmA(G)$ is possible in principle, the following homological approach reaches a cleaner statement faster:

\begin{theorem}\label{thm:kernel}
The kernel~$\rmA(G)$ of the extension~\hbox{$\Gr\Pq(G)\to G$} sits in an exact sequence
\[
\rmH_2(G;\bbZ)\longrightarrow
\rmA(G)\longrightarrow
\rmB(G)\longrightarrow
\rmH_1(G;\bbZ)\longrightarrow
0
\]
of abelian groups, where~$\rmB(G)$ is the quotient of the free abelian group~$\bbZ\Cl(G)$  with basis the conjugacy classes of elements in the group~$G$ modulo the relations~\hbox{$n[a]=[a^n]$} for all~$a$ in~$G$ and all integers~$n$.
\end{theorem}

\begin{proof}
Given any extension~$A\to E\to G$ of groups with abelian kernel~$A$, the five-term exact sequence~\cite[II.5,~Ex.~6]{Brown} gives an exact sequence
\[
\rmH_2(G;\bbZ)\longrightarrow
\rmH_0(G;A)\longrightarrow
\rmH_1(E;\bbZ)\longrightarrow
\rmH_1(G;\bbZ)\longrightarrow
0
\]
of abelian groups. When the extension is central, as in the case that we are interested in, the~$G$--action on~$A$ is trivial, and for the co-invariants we have~\hbox{$\rmH_0(G;A)=A$}. It remains to compute~$\rmH_1(E;\bbZ)$ for the group~\hbox{$E=\Gr\Pq(G)$}. Since the first homology is abelianization, this is~$\Ab\Gr\Pq(G)$. From Lemma~\ref{lem:adjoints_commute}, applied to~$P=\Pq(G)$, we get the description of that group as the~$\rmB(G)$ in the statement.
\end{proof}

\begin{remark}
The image of the Schur multiplier~$\rmH_2(G;\bbZ)$ inside the group~$\rmA(G)$ is clearly given by~$\rmA(G)\cap[E,E]$, where~\hbox{$E=\Gr\Pq(G)$} is as in the preceding proof. This description echoes Hopf's formula~\hbox{$\rmH_2(G;\bbZ)=(R\cap [F,F])/[R,F]$} when the group~$G$ is presented as a quotient~$G=F/R$ with~$F$ a free group:~\hbox{$[\rmA(G),E]=0$} because~$\rmA(G)$ is central in~$E$.
\end{remark}

\begin{corollary}\label{cor:finite}
If~$G$ is finite, so is~$\Gr\Pq(G)$.
\end{corollary}

\begin{proof}
If~$G$ is finite, so is~$\rmB(G)$. This is clear because~$\rmB(G)$ is a finitely generated abelian group whose elements are of finite order. Since also the homology groups~$\rmH_n(G;\bbZ)$ are finite for all~$n\geqslant 1$, we infer that~$\rmA(G)$ is finite as well, and then so must be~$\Gr\Pq(G)$.
\end{proof}


\section{The essential image}\label{sec:image}

This final section contains our results on the essential image of the functor~$\Gr$.

\begin{definition}\label{def:pq-gen}
Let~$G$ be a group. We call it {\it pq-generated} when there is a power quandle~$P$ such that~$G\cong\Gr(P).$
\end{definition}

Intuitively, a group is pq-generated if it can be presented using generators and relations such that all relations can be expressed in terms of conjugation and powers alone.

\begin{remark}
If a power quandle~$P$ as in Definition~\ref{def:pq-gen} exists, there is another such one that is a power subquandle of~$\Pq(G)$. We omit the proof, since we will not use this remark. It will be clear in the examples below.
\end{remark}

\begin{example}
Finite abelian groups are pq-generated. This is clear since these are products of finite cyclic groups, which are presentable as~$\langle a\,|\, a^n=e\rangle$, and the fact that~$a$ and~$b$ commute can be written as~$aba^{-1} = b$. A power subquandle~$P$ is given by the powers of the generators of the factors for some chosen product decomposition. 
\end{example}

\begin{example}
Coxeter groups are pq-generated. If~$S$ is a generating set, the Coxeter relations are~$s^2=e$ and otherwise take the form~$(st)^m=e$ for some~$m\geqslant 2$. The latter can be rewritten as~$aba^{-1}=t$ with~$b\in\{s,t\}$ and~$a=ststst...$ with~$m-1$ alternating factors. For instance, a power subquandle~$P$ for the symmetric groups is given by the conjugacy class of the transpositions and the unit--the elements that move at most two elements.
\end{example}

\begin{theorem}\label{thm:pq-gen}
Finite pq-generated groups with isomorphic underlying power quandles are isomorphic.
\end{theorem}

\begin{proof}
Let~$G\cong\Gr(P)$ for some power quandle~$P$. By abstract non-sense, we can apply the left-adjoint~$\Gr$ to the unit~$P\to\Pq\Gr(P)$ of the adjunction to get a morphism~\hbox{$\Gr(P)\to\Gr\Pq\Gr(P)$} of groups that splits the extension given by the co-unit~$\Gr\Pq\Gr(P)\to\Gr(P)$. Our Theorem~\ref{thm:central} now implies that~\hbox{$\Gr\Pq(G)\cong G\times\rmA(G)$}.

Therefore, if~$G$ and~$H$ are two pq-generated groups, an isomorphism~\hbox{$\Pq(G)\cong\Pq(H)$} of power quandles immediately induces an isomorphism~$\Gr\Pq(G)\cong\Gr\Pq(H)$ of groups, and by what we have shown in the first paragraph of the proof, an isomorphism
\begin{equation}\label{eq:cancel_1}
G\times\rmA(G)\cong H\times\rmA(H).
\end{equation}
Next, recall that we can cancel finite factors in products of groups~\cite{Jonsson--Tarksi}. Passing from~\eqref{eq:cancel_1} to centers, we get
\begin{equation}\label{eq:cancel_2}
\rmZ(G)\times\rmA(G)\cong\rmZ(H)\times\rmA(H).
\end{equation}
Recalling~$\rmZ(G)\cong\rmZ(H)$ from Theorem~\ref{thm:center}, we can cancel this finite factor in~\eqref{eq:cancel_2} and get~\hbox{$\rmA(G)\cong\rmA(H)$}. Then we can cancel this factor in~\eqref{eq:cancel_1}, because it is finite by Corollary~\ref{cor:finite}, and get~\hbox{$G\cong H$}, as claimed.
\end{proof}

\begin{question}\label{q:pq-gen}
Are all finite groups pq-generated?
\end{question}

By Theorem~\ref{thm:pq-gen}, a positive answer to this question would immediately imply a negative answer to Question~\ref{q:forgetful}, and conversely.



\end{document}